\documentclass[12pt]{amsart}
\usepackage{amscd,amsmath,amsthm,amssymb,graphics}
\usepackage{amsfonts,amssymb,amscd,amsmath,enumerate,verbatim,calc,latexsym,pstricks,pst-plot,pst-3d,}
\usepackage{amsmath}
\usepackage{amssymb}
\usepackage{amsbsy}
\usepackage{graphicx}
\usepackage{amsfonts}

\usepackage{amscd,amsmath,amsthm,amssymb}
\usepackage{pstricks,pst-plot,pst-3d}%\usepackage[T1]{fontenc}
\usepackage{lmodern,pst-node}%\usepackage{geometry}
\usepackage{pstricks}

\makeatletter

\newcommand{\Rmnum}[1]{\expandafter\@slowromancap\romannumeral #1@}
\makeatother

\usepackage{pgf,tikz}
\usetikzlibrary{arrows}

\def\frk{\frak}               % font for "Fraktur"

\def\Phi{{\frk n}}
\def\Phi{{\frk N}}
\def\opn#1#2{\def#1{\operatorname{#2}}} % to make operators
\opn\chara{char} \opn\length{\ell} \opn\pd{pd} \opn\rk{rk}
\opn\projdim{proj\,dim} \opn\injdim{inj\,dim} \opn\rank{rank}
\opn\depth{depth} \opn\grade{grade} \opn\height{height}
\opn\embdim{emb\,dim} \opn\codim{codim}

\opn\Tr{Tr} \opn\bigrank{big\,rank}
\opn\superheight{superheight}\opn\lcm{lcm}
\opn\trdeg{tr\,deg}%\emph{
\opn\reg{reg} \opn\lreg{lreg} \opn\ini{in} \opn\lpd{lpd}
\opn\size{size}\opn\bigsize{bigsize}
\opn\cosize{cosize}\opn\bigcosize{bigcosize}
\opn\sdepth{sdepth}\opn\sreg{sreg}
\opn\link{link}\opn\fdepth{fdepth}
\opn\index{index}
\opn\index{index}
\opn\indeg{indeg}
\opn\N{N}
\opn\SSC{SSC}
\opn\SC{SC}
\opn\conv{conv}
\opn\div{div} \opn\Div{Div} \opn\cl{cl} \opn\Cl{Cl}
\opn\Spec{Spec} \opn\Supp{Supp} \opn\supp{supp} \opn\Sing{Sing}
\opn\Ass{Ass} \opn\Min{Min}\opn\Mon{Mon} \opn\dstab{dstab} \opn\astab{astab}
\opn\Syz{Syz}
\opn\reg{reg}
\opn\Ann{Ann} \opn\Rad{Rad} \opn\Soc{Soc}
\opn\Im{Im} \opn\Ker{Ker} \opn\Coker{Coker} \opn\Am{Am}
\opn\Hom{Hom} \opn\Tor{Tor} \opn\Ext{Ext} \opn\End{End}
\opn\Aut{Aut} \opn\id{id}

\opn\nat{nat}
\opn\pff{pf}%   \pf exists already
\opn\Pf{Pf} \opn\GL{GL} \opn\SL{SL} \opn\mod{mod} \opn\ord{ord}
\opn\Gin{Gin} \opn\Hilb{Hilb}\opn\sort{sort}
\opn\initial{init}
\opn\ende{end}
\opn\height{height}
\opn\type{type}
\opn\aff{aff} \opn\con{conv} \opn\relint{relint} \opn\st{st}
\opn\lk{lk} \opn\cn{cn} \opn\core{core} \opn\vol{vol}
\opn\link{link} \opn\star{star}\opn\lex{lex}\opn\Mon{Mon}\opn\Min{Min}
\opn\gr{gr}

\def\pot#1#2{#1[\kern-0.28ex[#2]\kern-0.28ex]}

\opn\dirlim{\underrightarrow{\lim}}
\opn\inivlim{\underleftarrow{\lim}}
\def\Implies{\ifmmode\Longrightarrow \else
        \unskip${}\Longrightarrow{}$\ignorespaces\fi}
\def\implies{\ifmmode\Rightarrow \else
        \unskip${}\Rightarrow{}$\ignorespaces\fi}
\def\iff{\ifmmode\Longleftrightarrow \else
        \unskip${}\Longleftrightarrow{}$\ignorespaces\fi}

\let\:=\colon
\newtheorem{Theorem}{Theorem}[section]
 \newtheorem{Lemma}[Theorem]{Lemma}
 \newtheorem{Corollary}[Theorem]{Corollary}
 \newtheorem{Proposition}[Theorem]{Proposition}

 \newtheorem{Example}[Theorem]{Example}
 
 \newtheorem{Definition}[Theorem]{Definition}
 \newtheorem*{Definition*}{Definition}
 
 \newtheorem{Conjecture}{Conjecture}
 \newtheorem*{Conjecture*}{Conjecture}

\let\epsilon\varepsilon
\let\kappa=\varkappa
\def\qed{\ifhmode\textqed\fi
      \ifmmode\ifinner\quad\qedsymbol\else\dispqed\fi\fi}
\def\textqed{\unskip\nobreak\penalty50
       \hskip2em\hbox{}\nobreak\hfil\qedsymbol
       \parfillskip=0pt \finalhyphendemerits=0}
\def\dispqed{\rlap{\qquad\qedsymbol}}

\opn\dis{dis}
\def\pnt{{\raise0.5mm\hbox{\large\bf.}}}

\opn\Lex{Lex}

\renewcommand{\qedsymbol}{\hfill \ensuremath{\square}}

\begin{document}

\title{Some Results on $\mathrm{v}$-Number of Monomial Ideals}

\author{Liuqing Yang}
\address{Department of Mathematics, Soochow University, Suzhou, 215006, China}
\email{20214007001@stu.suda.edu.cn}

\author{Kaiwen Hu}
\address{Department of Mathematics, Soochow University, Suzhou, 215006, China}
\email{hukaiwen6@outlook.com}

\author{Lizhong Chu*}
\footnote{* Corresponding author}
\address{Department of Mathematics, Soochow University, Suzhou, 215006, China}
\email{chulizhong@suda.edu.cn}

\begin{abstract}
This paper investigates the v-number of various classes of monomial ideals. First, we considers the relationship between the v-number and the regularity of the mixed product ideal $I$, proving that $\mathrm{v}(I) \leq \mathrm{reg}(S/I)$. Next, we investigate an open conjecture on the v-number: if a monomial ideal $I$ has linear powers, then for all $k \geq 1$, $\mathrm{v}(I^k) = \alpha(I)k - 1.$ We prove that if a monomial ideal $I$ with linear powers is a homogeneous square-free ideal and ($k \geq 1$) has no embedded associated primes, then $\mathrm{v}(I^k) = \alpha(I)k - 1.$ We have also drawn some conclusions about the k-th power of the graph.Additionally,  we calculate the v-number of various powers of edge ideals(including ordinary power ,square-free powers, symbolic powers).

Finally, we propose a conjecture that the v-number of ordinary powers of line graph is equal to the v-number of square-free powers.
\end{abstract}


\maketitle

\section{Introduction}

Inspired by algebraic coding theory, particularly the concept of generalized Hamming weight (GHD), Susan M. Cooper and Rafael H. Villarreal investigated certain graded ideals in polynomial rings, especially Geramita ideals (which are non-mixed, one-dimensional, and generated by linear forms).Their goal was to use commutative algebra methods to study the GMD (Generalized Minimum Distance) function and analyze its behavior as various parameters change (\cite{cooper2020generalized} \cite{grisalde2021induced} \cite{martinez2017minimum} \cite{nunez2017footprint}). To explain the regularity of the minimum distance function $(\delta_I(d,r))$, 
it is important to note that the regularity $\mathrm{reg}(\delta_I)$ is difficult to calculate directly , Therefore, an algebraic tool is needed to characterize it.  The v-number provides an algebraic-geometric perspective for characterizing $\mathrm{reg}(\delta_I)$, and it has been proven that:
$$\mathrm{reg}(\delta_I) = \mathrm{v}(I).$$
Thus, the regularity of the minimum distance function is equal to the v-number,  allowing the stable points of $\delta_I$  to be calculated via the v-number. The concept of the v-number for graded ideals I is defined as follows:
\[
    \mathrm{v}(I) :=
    \begin{cases}
            \mathrm{min}\{ d\geq 1 \,|\, \exists f\in S_d \,\text{and}\, \mathrm{p}\in \mathrm{Ass}(I) \,\text{such that}\, (I:f) = \mathfrak{p} \}, & \text{if}\, I\subsetneq \mathfrak{m} \\
            0, & \text{if}\, I = \mathfrak{m}
    \end{cases}
\]
where $\mathrm{Ass}(I)$ is the set of associated primes of $S/I$ and $\mathfrak{m} = (x_1, \dots, x_n)$ is the maximal ideal of S.Moreover, Rafael H. Villarreal et al. discovered that the v-number not only characterizes the regularity of the GMD function, but also describes the minimum degree of elements that can decompose the ideal $I$. The v-number connects the Hilbert function (which describes the growth behavior of ideals) with the minimum distance function (which describes characteristics of coding). The results of this paper demonstrate that the v-number is both an algebraic invariant and a important parameter in describing the behavior of the minimum distance function.

Dipankar Ghosh extended the concept of the v-number from graded ideals to graded modules over graded Noetherian rings. Furthermore, they interpreted this algebraic invariant as the initial degree of graded modules and proved several results(\cite{fiorindo2024asymptotic}). 

Based on the aforementioned definition of the v-number, Delio Jaramillo and Rafael H. Villarreal provided a method to calculate the v-number of edge ideals of clutter graphs, primarily using stable sets for the calculation(\cite{jaramillo2021v}). Following this method, Prativa Biswas and Mousumi Mandal calculated the v-number of edge ideals for various types of graphs, such as line graphs, cycle graphs, and others(\cite{biswas2024study}).  In this paper, we also adopt the important parameter to calculate the v-number of edge ideals for various classes of graph.
Civan, based on the definition of the v-number for graded ideals, extended the concept to define the v-number for simplicial complexes(\cite{civan2023v}), Using this definition, we calculate the v-number for a class of covering ideals. 

D. Grayson and M. A. Henning provided a method for calculating the regularity, v-number, and dimension of the edge ideal of a graph using Macaulay2. Delio Jaramillo and Rafael H. Villarreal proved that, in certain special cases, $\mathrm{v}(I(G))\leq \mathrm{reg}(S/I(G))$, and they posed an open problem: Is it true that for any square-free monomial ideal $I$, $\mathrm{v}(I)\leq \mathrm{reg}(S/I) + 1$(\cite{jaramillo2021v}). Kamalesh Saha and Indranath Sengupta found several classes of disconnected graphs for which the edge ideal satisfies $\mathrm{v}(I) > \mathrm{reg}(S/I) + 1$(\cite{saha2022v}). Yusuf Civan proved that for any positive integer $k\geq 1$, there exists a connected graph $H_k$ such that $\mathrm{v}(I(H_k)) = \mathrm{reg}(I(H_k)) + k$(\cite{civan2023v}).  In this paper, we prove that for mixed product ideals, $\mathrm{v}(I) \leq \mathrm{reg}(S/I)$.

Antonino Ficarra proposed an open conjecture(\cite{ficarra2023simonconjecturetextvnumbermonomial}).

\noindent
\textbf{Conjecture}:\,Let $I\subset S$ be a monomial ideal with linear powers. Then, for any $k\geq 1$, 
$$\mathrm{v}(I^k) = \alpha(I)k - 1.$$
where $\alpha(I)$ denotes the degree of the monomial with the smallest degree in $I$, and the linear power condition implies that the ideal $I$ and its k-th power have linear resolutions.  In this paper, we prove that this conjecture holds for homogeneous monomial ideals with linear powers, provided that $I^k(k\geq 1)$ contains no embedded associated primes.References \cite{ficarra2023simonconjecturetextvnumbermonomial} \cite{saha2022v} \cite{ficarra2024v} \cite{ficarra2023asymptotic} \cite{biswas2024asymptotic}explore the v-number of powers of various types of ideals, and based on these results, we establish relationships between the v-number of powers of edge ideals for graphs.  

References \cite{sarkar2024v} \cite{saha2024binomialexpansionmathrmvnumber} \cite{fiorindo2024asymptotic} 
consider the v-number of symbolic powers of various types of ideals. Building on this, we derive relationships between the v-numbers of symbolic powers of homogeneous, square-free monomial ideals and calculate the v-number for the symbolic powers of the edge ideal of a complete graph.Inspired by symbolic powers, we also calculate the v-number of several monomial ideals using the concept of square-free powers.

\newpage
\section{Preliminaries}

This chapter introduces some of the definitions and basic results required for the present work. Let \( S = K[x_1, \dots, x_n] = \bigoplus_{d=0}^{\infty} S_d \) be a standard graded polynomial ring over a field \( K \). In particular, in some cases, let \( S_1 = K[x_1, \dots, x_n] = K[\mathbf{x}] \), \( S_2 = K[y_1, \dots, y_m] = K[\mathbf{y}] \), and \( S = K[x_1, \dots, x_n, y_1, \dots, y_m] = K[\mathbf{x}, \mathbf{y}] \) be polynomial rings over the field \( K \).

\begin{Definition}
    (Simple Graph): A simple graph is a graph that does not contain loops (edges that connect a vertex to itself) or multiple edges (more than one edge between two vertices).A simple graph $G$ can be represented by an ordered pair $G = (V,E)$, where $V$ is the set of vertices, representing the points in the graph, and  $E\subseteq\{ \{ u,v \} | u,v\in V, u\neq v \}$ is the set of edges, representing the relationships connecting the vertices.The edges are undirected, meaning the graph is an undirected graph.
    
    This work primarily considers simple undirected graphs.
\end{Definition}

\begin{Definition}
    A graph is said to be connected if there exists a path from any vertex to every other vertex in the graph. 
\end{Definition}

\begin{Definition}
    A hypergraph is a generalization of a graph that allows edges to connect more than two vertices. A hypergraph is represented as $H = (V,E)$, where $V$ is the set of vertices, representing the points in the graph, and $E\subseteq \mathcal{P}(V)\backslash \emptyset$ is the set of hyperedges, where each hyperedge is a subset of vertices. $\mathcal{P}(V)$ is the power set of $\mathcal{P}(V)$ (the set of all subsets) of $V$.
    
    A clutter graph is a special type of hypergraph, usually denoted as $\mathcal{C} = (V(\mathcal{C}),$ $E(\mathcal{C}))$, in which no hyperedge is a subset of another hyperedge. In other words, if $e_1,e_2 \in E(\mathcal{C})$, then $e_1\not\subset e_2$.
\end{Definition}

\begin{Definition}
    A stable set (or independent set) in a clutter graph is denoted by $A$, and satisfies the condition that for every $e\in E(\mathcal{C}),e \not\subset A$. 

    A vertex cover of a clutter graph is denoted by $C$, and satisfies the condition that for every $e\in E(\mathcal{C}), e \cap C \neq \varnothing$. \textbf{Stable sets} and \textbf{vertex covers} are dual concepts, meaning that $C$ is a vertex cover of the clutter graph if and only if $V(\mathcal{C})\backslash C$ is a stable set.   
\end{Definition}

\begin{Definition}
    A minimal vertex cover of a clutter graph is denoted by $C$ and satisfies the condition that for every $v\in C$, the set $C' = C\backslash \{ v \}$ is no longer a vertex cover of the clutter graph. A maximal stable set is denoted by $A$ and satisfies the condition that for every $v\notin A$, the set $A' = A\cup \{ v \}$ is no longer a stable set of the clutter graph.
\end{Definition}

\begin{Definition}
    The neighbor set of a stable set $A$ in a clutter graph, denoted by $N_{\mathcal{C}}(A)$, is the set of all vertices $v$ such that $A\cup \{ v \}$ contains at least one hyperedge of $\mathcal{C}$. 
\end{Definition}

In this paper, we define the set $\mathcal{A}_{\mathcal{C}}$ as the set of stable sets in the clutter graph $\mathcal{C}$ that correspond to minimal vertex covers composed of the neighbor sets of $\mathcal{C}$.

\begin{Definition}
    A matching in a graph $G$ is a set of edges $\{ e_1,\dots, e_s \}\subseteq E(G)$ such that no two edges share a common vertex. The number of edges in the maximum matching of $G$ is called the matching number, denoted by $\mathrm{match}(G)$.
\end{Definition}

\begin{Definition}
Let $G$ be a graph, and let $u$ and $v$ be two vertices (which may be the same). If there exists a vertex sequence $p_0,p_1,\dots, p_{2r+1} (r\geq 1)$ of $G$ that satisfies the following condition. 
\begin{itemize}
    \item[(1)] $p_0 = u, p_{2r+1} = v;$
    \item[(2)] $p_0p_1, p_1p_2, \dots, p_{2r}p_{2r+1}$are edges of the graph $G$;
    \item[(3)] $\forall 1\leq k\leq r-1,\, \exists\, i, \{ p_{2k+1, p_{2k+2}} = e_i \} $;
    \item[(4)] $\forall i, |\{ k: \{ p_{2k+1, p_{2k+2}} \} = e_i \}| \leq | \{ j: e_i = e_j \} |$.
\end{itemize}
Then, $u$ and $v$ are said to be \textbf{even connected} with respect to the product of edges $e_1e_2\dots e_s$ in $G$.
\end{Definition}

\begin{Lemma}[\cite{fakhari2024castelnuovo},Corollary 3.4 ]
\label{H}
Let $G$ be a graph, and let $s\leq \mathrm{math}(G) - 1$ be a positive integer. Suppose $\mu = e_1\dots e_s$ is a monomial in $I(G)^{[s]}$. Then, there exists a simple graph $H$ with the vertex set $V(H) = V(G)\backslash \mathrm{supp}(\mu)$, and $I(H) = (I(G)^{[s+1]}:\mu)$, where $\mathrm{supp}(\mu)$ is the set of variables appearing in the monomial $\mu$. Moreover, any two vertices $x_i,x_j$ in $V(H)$ are connected in $H$ if and only if one of the following conditions holds: 
\begin{itemize}
    \item[(1)] $x_i$ and $x_j$ are connected in $G$;
    \item[(2)] $x_i$ and $x_j$ are even connected in $G$ with respect to $e_1, \dots, e_s$.
\end{itemize}
\end{Lemma}

\begin{Definition}
    Let $I$ be an ideal of $S$, and let  $I = \mathfrak{q}_1\cap \mathfrak{q}_2\cap \dots \cap \mathfrak{q}_n$ be a reduced prime decomposition of $I$, with $\mathfrak{p}_i = \sqrt{\mathfrak{q}_i}, i= 1,\dots, n.$. Then:
 $$\{ \mathfrak{p}_1, \dots, \mathfrak{p}_n \} = \mathrm{Ass}(S/I). $$
\end{Definition}
In this paper, the associated prime ideals of an ideal $\mathrm{Ass}(I)$ refer to the associated primes of $\mathrm{Ass}(S/I)$.

\begin{Definition}
The v-number of a graded ideal $I$ is denoted as $\mathrm{v}(I)$ and is given by:
\[
    \mathrm{v}(I) :=
    \begin{cases}
        \mathrm{min}\{ d\geq 1 \,|\, \exists f\in S_d \,\text{and}\, \mathrm{p}\in \mathrm{Ass}(I) \,\text{such that}\, (I:f) = \mathfrak{p} \}, & \text{if}\, I\subsetneq \mathfrak{m} \\
        0, & \text{if}\, I = \mathfrak{m}
    \end{cases}
\]
where $\mathrm{Ass}(I)$ is the set of associated primes of $S/I$, and $\mathfrak{m} = (x_1, \dots, x_n)$ is the maximal ideal of S. The monomial f corresponding to the v-number of I is any f that satisfies the definition above. 
\end{Definition}
The monomial $f$ corresponding to the v-number of $I$ is any $f$ that satisfies the definition above.

\begin{Lemma}[\cite{jaramillo2021v},Theorem 3.5]
\label{stable}
    Let $I(\mathcal{C})$ be the edge ideal of a clutter graph $\mathcal{C}$. Then, the v-number of $I(\mathcal{C})$ is: 
$$\mathrm{v}(I(\mathcal{C})) = \mathrm{min}\{ |A|: A\in \mathcal{A}_{\mathcal{C}} \}, $$
where $\mathcal{A}_{\mathcal{C}} \}$ denotes the set of stable sets in $\mathcal{C}$.
\end{Lemma}

In this paper, the v-number of the ideal $I(\mathcal{C})$ corresponds to a stable set $A$ , which is any set 
$A$ that satisfies the above definition.
    
\begin{Theorem}[\cite{villarreal2001monomial}, Lemma 6.3.37]
    Let $I = I(\mathcal{C})$ be the edge ideal of a clutter graph $\mathcal{C}$. Then, $\mathfrak{p}\in \mathrm{Ass}(I)$ if and only if $\mathfrak{p} = \langle C \rangle$ is a minimal vertex cover of the clutter graph $\mathcal{C}$.
\end{Theorem}

\begin{Definition}[\cite{peeva2010graded}]
The polarization of a monomial $x_i^{a_i}$ is defined as: 
$$x_i^{a_i} (pol) = \prod_{j=1}^{a_i} x_{i,j}, $$
hence $\mathbf{x}^{\alpha} = x_1^{a_1}\dots x_n^{a_n}$ is defined as:
$$\mathbf{x}^{\alpha}(pol) = x_1^{a_1}(pol)\dots x_n^{a_n}(pol).$$
Then, for any monomial ideal $I= \langle \mathbf{x}^{\alpha_1}, \dots, \mathbf{x}^{\alpha_n} \rangle \subseteq S$, the polarization of $I$ is defined as:
$$I(pol)= \langle \mathbf{x}^{\alpha_1}(pol), \dots, \mathbf{x}^{\alpha_n}(pol) \rangle, $$
which is a square-free monomial ideal in the ring $S(pol) = K[x_{i,j} | 1\leq i\leq n, 1\leq j\leq r_i]$, where $r_i$ is the maximum exponent of $x_i$ in the generators of $I$.
\end{Definition}

\begin{Lemma}[\cite{saha2022v}, Corollary 3.5]
\label{polideal}
    For a monomial ideal $I$, we have $\mathrm{v}(I(pol))\leq \mathrm{v}(I)$. Furthermore, if I has no embedded prime ideals, then $\mathrm{v}(I(pol)) = \mathrm{v}(I). $
\end{Lemma}
Here, $I$ has no embedded prime ideals means that there is no inclusion relationship among the associated prime ideals of $I$.

\begin{Lemma}[\cite{saha2022v}, Proposition 3.7]
\label{alpha}
Let I be a monomial ideal, and let $\mathfrak{p} = (x_{s_1,b_{s_1}}, \dots, x_{s_k,b_{s_k}} )$ $\in \mathrm{Ass}(I(pol))$, such that $\mathrm{v}(I(pol)) = \mathrm{v}_{\mathfrak{p}}(I(pol))$, and I has no embedded prime ideals containing $(x_{s_1}, \dots, x_{s_k})$. Then:
$$\mathrm{v}(I) = \mathrm{min}\{ \alpha((I:\mathfrak{p})/I) | \mathfrak{p} \in \mathrm{Ass}(I) \}. $$
\end{Lemma}

\begin{Lemma}[\cite{ficarra2023simonconjecturetextvnumbermonomial}]
     Let $I\subset S$ be a graded ideal with no embedded prime ideals. Then:
$$\mathrm{v}(I) = \mathrm{min}\{ \alpha((I:\mathfrak{p})/I) | \mathfrak{p} \in \mathrm{Ass}(I) \}. $$
\end{Lemma}

\begin{Theorem}
\label{additive}
     Let $S_1 = K[\mathbf{x}], S_2 = K[\mathbf{y}]$ be two polynomial rings over a field $K$, and let $S = K[\mathbf{x}, \mathbf{y}]$. If $I, J$ are square-free monomial ideals in $S_1, S_2$ respectively, then: 
$$\mathrm{v}(IS + JS) = \mathrm{v}(I) + \mathrm{v}(J).$$
\end{Theorem}

Similar conclusions can be drawn as follows.

\begin{Proposition}
\label{IJ}
    Let \( S_1 = K[x_1, x_2, \dots, x_n] \) and \( S_2 = K[y_1, y_2, \dots, y_m] \) be two polynomial rings over a field \( K \). If \( I \) and \( J \) are square-free monomial ideals in \( S_1 \) and \( S_2 \), respectively, and \( S = K[\mathbf{x}, \mathbf{y}] \), then
    \[
    \mathrm{v}(IJ) = \min\{ \alpha(I) + \mathrm{v}(J), \alpha(J) + \mathrm{v}(I) \}.
    \]
\end{Proposition}
\begin{proof}
    Without loss of generality, let \( \mathfrak{p} \) be a minimal associated prime ideal of \( IJ \) such that \( (IJ : \mathfrak{p}) = f \), and \( \mathrm{deg}(f) = \mathrm{v}(IJ) \). By the definition of associated primes, \( \mathrm{Ass}(IJ) = \mathrm{Ass}(I) \cup \mathrm{Ass}(J) \). Thus, \( \mathfrak{p} \) has the following two cases:

    \textbf{Case 1}: If \( \mathfrak{p} = \mathfrak{p}_1 \in \mathrm{Ass}(I) \), then
    \[
    (IJ : \mathfrak{p}) \setminus IJ = (I : \mathfrak{p}_1) J \setminus IJ,
    \]
    and thus \( \mathrm{v}_{\mathfrak{p}}(IJ) = \mathrm{v}(I) + \alpha(J) \).

    \textbf{Case 2}: If \( \mathfrak{p} = \mathfrak{p}_2 \in \mathrm{Ass}(J) \), then
    \[
    (IJ : \mathfrak{p}) \setminus IJ = (J : \mathfrak{p}_2) I \setminus IJ,
    \]
    and thus \( \mathrm{v}_{\mathfrak{p}}(IJ) = \mathrm{v}(J) + \alpha(I) \).

    In conclusion, we have
    \[
    \mathrm{v}(IJ) = \min\{ \alpha(I) + \mathrm{v}(J), \alpha(J) + \mathrm{v}(I) \}.
    \]
\end{proof}

\begin{Lemma}[\cite{ficarra2023asymptotic}, Corollary 2.9]
\label{inftyprime}
    Let \( S \) be a Noetherian integral domain, \( I \subset R \) be an ideal, and \( \mathfrak{p} \in \mathrm{Ass}^{\infty}(I) \) be a stable prime ideal of \( I \). Then, for any \( k \gg 0 \),
    \[
    (I^k : \mathfrak{p}) = I(I^{k-1} : \mathfrak{p}).
    \]
\end{Lemma}

\begin{Lemma}[\cite{biswas2024study}, Theorem 3.1]
\label{line}
    Let \( L_n \) be a path graph with \( n \) vertices. Then,
    \[
    \mathrm{v}(I(L_n)) =
    \begin{cases}
    \left\lfloor \frac{n}{4} \right\rfloor, & \text{if } n \equiv 0,1 \ (\mathrm{mod} \ 4) \\
    \left\lfloor \frac{n}{4} \right\rfloor + 1, & \text{if } n \equiv 2,3 \ (\mathrm{mod} \ 4).
    \end{cases}
    \]
\end{Lemma}

In this paper, the notation \( [x] \) (where \( x \in \mathbb{R} \)) denotes the smallest integer greater than or equal to \( x \).

\begin{Definition}
    Let \( I \subset S \) be a graded ideal, and let \( \textbf{F} \) be a minimal graded free resolution of \( S/I \) as an \( S \)-module:
    \[
    \textbf{F}:\; 0 \rightarrow \bigoplus_j S(-j)^{b_{g,j}} \rightarrow \dots \rightarrow \bigoplus_j S(-j)^{b_{1,j}} \rightarrow S \rightarrow S/I \rightarrow 0.
    \]
    Then, the Castelnuovo-Mumford regularity of \( S/I \) is defined as
    \[
    \mathrm{reg}(S/I) = \max\{ j - i \mid b_{i,j} \neq 0 \}.
    \]
\end{Definition}
    
\begin{Definition}
    Let \( S_1 = K[\mathbf{x}] \), \( S_2 = K[\mathbf{y}] \), and \( S = K[\mathbf{x}, \mathbf{y}] \) be polynomial rings over a field \( K \). A mixed product ideal is defined as
    \[
    \sum_{i=1}^{s} I_{q_i} J_{r_i}, \quad q_i, r_i \in \mathbb{N}^*,
    \]
    where \( I_{q_i} \) (resp. \( J_{r_i} \)) is the square-free monomial ideal generated by all variables in \( \mathbf{x} \) (resp. \( \mathbf{y} \)) of degree \( q_i \) (resp. \( r_i \)), and satisfies \( 1 \leq q_1 < q_2 < \dots < q_s \leq n \) and \( 1 \leq r_s < r_{s-1} < \dots < r_1 \leq m \).
\end{Definition}

\begin{Lemma}[\cite{rinaldo2012sequentially}, Corollary 3.3]
\label{mixedprime}
    Let \( S = K[\mathbf{x}, \mathbf{y}] \), and define \( \mathcal{X}_i = \{ X \subset \{ \mathbf{x} \} : |X| = i \} \), \( \mathcal{Y}_j = \{ Y \subset \{ \mathbf{y} \} : |Y| = j \} \). If \( i > n \) (resp. \( j > m \)), then \( \mathcal{X}_i = \varnothing \) (resp. \( \mathcal{Y}_j = \varnothing \)). Then,
    \[
    \sum_{i=1}^{s} I_{q_i} J_{r_i} = \mathcal{P}_x \cap \mathcal{P}_{xy} \cap \mathcal{P}_y,
    \]
    where
    \[
    \mathcal{P}_x = \bigcap_{X \in \mathcal{X}_{n-q_1+1}} (X), \quad \mathcal{P}_y = \bigcap_{Y \in \mathcal{Y}_{m-r_s+1}} (Y),
    \]
    and
    \[
    \mathcal{P}_{xy} = \bigcap_{i=1}^{s-1} \left( \bigcap_{X, Y} ((X) + (Y)) \right), \quad X \in \mathcal{X}_{n-q_i+1}, Y \in \mathcal{Y}_{m-r_i+1}.
    \]
\end{Lemma}

\begin{Definition}
    Let \( S \) be a standard graded polynomial ring, and let \( I \subset S \) be a monomial ideal. The square-free powers of \( I \), denoted by \( I^{[k]} \), are defined as the ideal generated by all square-free monomials in \( I^k \).
\end{Definition}
    
By the definition of square-free powers, for the edge ideal \( I(G) \) of a graph \( G \), the degree \( k \) must satisfy
\[
1 \leq k \leq \mathrm{match}(G),
\]
where \( \mathrm{match}(G) \) is the matching number of \( G \).

\begin{Lemma}[\cite{crupi2023matchings}, Proposition 2.3]
\label{forest}
    Let \( G \) be a forest graph with \( n \) vertices, \( G_1 \) be the induced subgraph on the vertex set \( [n-1] \), and \( G_2 \) be the induced subgraph on the vertex set \( [n-2] \). Then, for any \( 1 \leq k \leq \mathrm{match}(G) \),
    \[
    I(G)^{[k]} = I(G_1)^{[k]} + x_{n-1} x_n I(G_2)^{[k-1]}.
    \]
\end{Lemma}

\newpage

\section{The $\mathrm{v}$-number of Mixed Product Ideals}
Based on the structure of the associated prime ideals of mixed product ideals introduced in Lemma \ref{mixedprime}, this chapter calculates the $\mathrm{v}$-$\mathrm{number}$ of mixed product ideals and further analyzes its relationship with the regularity. 

\begin{Theorem}
\label{mixed}
Let the mixed product ideal $I_{q_i}$(respectively, $J_{r_j}$) ($1\leq i,j\leq s\in \mathbb{N}^*$) be the square-free monomial ideal generated by all variables of degree $q_i$(respectively, $r_j$) in the polynomial ring $S_1=K[\textbf{x}]$(respectively, $S_2=K[\textbf{y}]$) over the field $K$, and let $S=K[\textbf{x}, \textbf{y}]$. Here, $1\leq q_1 < q_2 <\dots < q_s\leq n$ and $1\leq r_s < r_{s-1} <\dots < r_1\leq m$. In particular, $1\leq q < s\leq n $ and $1\leq t < r\leq m $. Then, 
\begin{itemize}
    \item[(1)] $\mathrm{v}(I_qS) = q-1;$
    \item[(2)] $\mathrm{v}(I_qJ_rS) = q+r-1;$
    \item[(3)] $\mathrm{v}((I_q + J_t)S) = q+t-2;$
    \item[(4)] $\mathrm{v}((I_qJ_r +I_s)S) = q+r-1;$
    \item[(5)] $\mathrm{v}((I_qJ_r +I_sJ_t)S) = \mathrm{min}\{ q+r-1 , s+t-1 \};$
    \item[(6)] $\mathrm{v}((\sum_{i=1}^{s} I_{q_i}J_{r_i})S) = \mathrm{min}\{ q_1+r_1-1, q_s+r_s-1, q_{i+1}+r_{i}-2(2\leq i\leq s-2) \}.$
\end{itemize}
\end{Theorem}
\begin{proof}
$(1)$\; Since $I_qS$ is an ideal generated by square-free monomials of degree $q$, it follows that $v(I_qS) \geq q-1$. Let $f = x_{n-q+2} \dots x_n$. Then, $deg(f)=q-1$, and $(I_qS:f) = \mathfrak{p} = \langle x_1, x_2,\dots, x_{n-q+1} \rangle$ is an associated prime ideal of $I_qS$. Therefore, $v(I_qS) \leq q-1$. Combining these results, we conclude:
$$v(I_qS) = q-1. $$

$(2)$\; Similarly, $v(I_qJ_rS) \geq q+r-1$. Let $f = x_{n-q+2}\dots x_n y_{m-r+1}\dots y_m$. Then, $deg(f) = q+r-1$, and $(I_qJ_tS : f) = \mathfrak{p} = \langle x_1,x_2,\dots ,x_{n-q+1} \rangle$ is an associated prime ideal of $I_qJ_tS$. Therefore, $\mathrm{v}(I_qJ_rS) \leq q+r-1$. Combining these results, we conclude:
$$\mathrm{v}(I_qJ_rS) = q+r-1. $$

$(3)$\; By Lemma \ref{additive}, we have
$$\mathrm{v}((I_q + J_t)S) = \mathrm{v}(I_qS) + \mathrm{v}(J_tS) = q + t - 2.$$

$(4)$\; Given that $I = (I_qJ_r + I_s)S $ has the primary decomposition:
$$I = [\underset{X \in \mathcal{X}_{n-q+1} }{\cap} (X)] \cap [\bigcap_{\substack{ X \in \mathcal{X}_{n-s+1} \\ Y \in \mathcal{Y}_{m-r+1} }} ((X)+(Y))], $$
let $f = f_xf_y$ such that $(I:f)=\mathfrak{p}$, where $\mathfrak{p}$ is an associated prime ideal of $I$, and $f_x$(respectively, $f_y$) is the square-free monomial consisting of the product of all variables $x_i$(respectively, $y_j$) in $f$. Therefore, the associated prime ideal $\mathfrak{p}$ can be divided into the following two cases: 

\textbf{Case1}: $\mathfrak{p}$ is of the form $(X)$, where $X \in \mathcal{X}_{n-q+1} $. Since $q < s$ and $\mathfrak{p} = (I:f)$, by the proof of (2), we can choose $f = x_{n-q+2}\dots x_n y_{m-r+1}\dots y_m$, and such an $f$ has the minimal degree. In this case, 
$$\mathrm{v}_{\mathfrak{p}}(I) = deg(f) = q+r-1. $$

\textbf{Case2}: $\mathfrak{p}$ is of the form $((X)+(Y))$, where $X \in \mathcal{X}_{n-s+1}$ and $Y \in \mathcal{Y}_{m-r+1}$. Then, the minimal degree $f=f_xf_y$ satisfying $\mathfrak{p} = (I:f)$ is given by: 
$$f_x = x_{n-s+2}\dots x_n, f_y = y_{m-r+2}\dots y_m, $$
and thus, the minimal degree of such an $f$ is: 
$$deg(f) = s+r-2 \geq q+r-1. $$ 

In summary, 
$$v((I_qJ_r + I_s)S) = deg(f)_{min}= q+r-1. $$ 

$(5)$\; Given that $I = (I_qJ_r +I_sJ_t)S$ has the primary decomposition: 
$$I = [\underset{X \in \mathcal{X}_{n-q+1} }{\cap} (X)] \cap [\bigcap_{\substack{ X \in \mathcal{X}_{n-s+1} \\ Y \in \mathcal{Y}_{m-r+1} }} ((X)+(Y))] \cap [\underset{Y \in \mathcal{Y}_{m-t+1} }{\cap} (Y)], $$
let $f = f_xf_y$. Then, the associated prime ideal $\mathfrak{p}$ can be divided into the following three cases: 

\textbf{Case1}: $\mathfrak{p}$ is of the form $(X)$, where $X \in \mathcal{X}_{n-q+1} $. Since $q < s$, we can choose the minimal degree polynomial: 
$$f = x_{n-q+2}\dots x_n y_{m-r+1}\dots y_m, $$
which satisfies $(I:f) = \mathfrak{p}$. Therefore, 
$$\mathrm{v}_{\mathfrak{p}}(I) = q+r-1. $$

\textbf{Case2}: $\mathfrak{p}$  is of the form $((X)+(Y))$, where $X \in \mathcal{X}_{n-s+1}$ and $Y \in \mathcal{Y}_{m-r+1} $. Then, the minimal degree $f=f_xf_y$ satisfying $\mathfrak{p} = (I:f)$ is given by: 
$$f_x = x_{n-s+2}\dots x_n, f_y = y_{m-r+2}\dots y_m. $$
Thus, the minimal degree of such an $f$ is: 
$$\mathrm{deg}(f) = r+s-2. $$

\textbf{Case3}: $\mathfrak{p}$ is of the form $(Y)$, where $Y \in \mathcal{Y}_{m-t+1} $. Since $t < r$, we can choose the minimal degree polynomial: 
$$f = x_{n-s+1}\dots x_n y_{m-t+2}\dots y_m, $$
which satisfies $(I:f) = \mathfrak{p}$. Therefore, 
$$\mathrm{v}_{\mathfrak{p}}(I) = s+t-1.$$

In summary, 
$$\mathrm{v}((I_qJ_r +I_sJ_t)S) = deg(f)_{min} = \mathrm{min}\{ q+r-1 , s+t-1 \}. $$ 

$(6)$\; Given that $I = (\sum_{i=1}^{s} I_{q_i}J_{r_i})S$ has the primary decomposition: 
$$I = [\underset{X \in \mathcal{X}_{n-q_1+1} }{\cap} (X)] \cap [\bigcap_{i=1}^{s-1}\bigcap_{\substack{ X \in \mathcal{X}_{n-q_{i+1}+1} \\ Y \in \mathcal{Y}_{m-r_i+1} }} ((X)+(Y))] \cap [\underset{Y \in \mathcal{Y}_{m-r_s+1} }{\cap} (Y)], $$
then, \textbf{Case1} and \textbf{Case3} are similar to the case in (5), and under the corresponding associated prime ideals $\mathfrak{p}_1$ and $\mathfrak{p}_3$, we have: 
$$\mathrm{v}_{\mathfrak{p_1}}(I) = q_1+r_1-1, \mathrm{v}_{\mathfrak{p_3}}(I) = q_s+r_s-1. $$

\textbf{Case2}: $\mathfrak{p}_2$ is of the form $((X)+(Y))$, where $X \in \mathcal{X}_{n-q_{i+1}+1}$ and $Y \in \mathcal{Y}_{m-r_i+1}$ $(1\leq i\leq s-1)$. Then, the minimal degree $f=f_xf_y$ satisfying $\mathfrak{p} = (I:f)$ is given by: 
$$f_x = x_{n-q_{i+1}+2}\dots x_n, f_y = y_{m-r_i+2}\dots y_m. $$
Thus, the minimal degree of such an $f$ is: 
$$\mathrm{deg}(f) = q_{i+1}+r_i-2(1\leq i\leq s-1).$$

In summary, 
$$\mathrm{v}((\sum_{i=1}^{s} I_{q_i}J_{r_i})S) = \mathrm{min}\{ q_1+r_1-1, q_s+r_s-1, q_{i+1}+r_{i}-2(2\leq i\leq s-2) \}.$$
\end{proof}

Regarding the regularity of mixed product ideals, the following results have been established (see the reference \cite{chu2017stanley}): 

\begin{Lemma}
\label{regg}
Let the mixed product ideals $I_q\subseteq S_1=K[\textbf{x}]$, $J_r\subseteq S_2=K[\textbf{y}]$, and $S= K[\textbf{x}, \textbf{y}]$, where $1\leq q < s \leq n$ and $1\leq t < r \leq m$. Then: 
\begin{itemize}
    \item[(1)] $\mathrm{reg}(I_qS) = q;$
    \item[(2)] $\mathrm{reg}((I_qJ_r)S) = q+r;$
    \item[(3)] $\mathrm{reg}((I_q + J_t)S) = q+t-1;$
    \item[(4)] $\mathrm{reg}((I_qJ_r + I_s)S) = s+r-1;$
    \item[(5)] $\mathrm{reg}((I_qJ_r + I_sJ_t)S) = s+r-1.$
\end{itemize}
\end{Lemma}

Based on Theorem \ref{mixed} and Lemma \ref{regg}, the relationship between the $\mathrm{v}$-$\mathrm{number}$ and the regularity of the mixed product ideal $I$ can be established. 

\begin{Theorem}
\label{reg}
Let the monomial ideal $I\subseteq S= K[\mathbf{x}, \mathbf{y}]$ be a mixed product ideal of the following form, where $1\leq q < s\leq n$ and $1\leq t < r\leq m $: 
\begin{itemize}
    \item[(1)] $I_qS;$
    \item[(2)] $I_qJ_rS;$
    \item[(3)] $(I_q + J_t)S;$
    \item[(4)] $(I_qJ_r +I_s)S;$
    \item[(5)] $(I_qJ_r +I_sJ_t)S.$
\end{itemize}
Then, 
$$\mathrm{v}(I)\leq \mathrm{reg}(S/I).$$
\end{Theorem}

\section{The $\mathrm{v}$-number of edge ideals}
Based on Rafael H. Villarreal's study of the $\mathrm{v}$-$\mathrm{number}$of edge ideals of clutters \cite{jaramillo2021v}, this section calculates the $\mathrm{v}$-$\mathrm{number}$ of the edge ideal $I(G)$ by identifying the stable sets $\mathcal{A}_G$ of the graph $G$, where $\mathcal{A}_G$ represents the collection of all stable sets whose neighborhood sets are minimal vertex covers of $G$ . Combining with the formula
$$\mathrm{v}(I(G)) = \mathrm{min}\{ |A|: A\in\mathcal{A}_G \},$$
the $\mathrm{v}$-$\mathrm{number}$ of the edge ideal $I(G)$ can be calculated.

Let $G$ be a simple graph, $I(G)$ its edge ideal, and $G^{k}$($k\in \mathbb{N}^{*}$) the $k$-th power of the graph $G$ \cite{iqbal2019depth} and \cite{iqbal2018depth}. The $k$-th power of the graph $G$, denoted by $G^k$, is defined as another graph with the same vertex set as $G$ , where two vertices in $G^k$ are connected if and only if their distance in $G$ is less than or equal to $k$. The edge ideal of $G^k$ is denoted by $I(G^k)$.
\begin{Theorem}
\label{Gd}
Let $G$ be a connected graph with $d$ vertices. Then, for any $1\leq k\leq d-1$, we have
$$\mathrm{v}(I(G^{d-1})) = 1 \leq \dots \leq \mathrm{v}(I(G^k)) \leq \dots \leq \mathrm{v}(I(G)).$$
\end{Theorem}
\begin{proof}
First, $I(G^{d-1})$ is the edge ideal of a complete graph. By choosing the stable set $A=\{ x_1 \}$, we observe that $N_{G^{d-1}}(A)$ is a minimal vertex cover of $G^{d-1}$. Therefore, $\mathrm{v}(I(G^{d-1})) = 1$. 

Let $A_k$ be the stable set corresponding to $\mathrm{v}(I(G^k))$, satisfying $\mathrm{v}(I(G^k)) = |A_k|$. For any $x_{k_i} \in A_k$, we have $|N_{G^{k-1}}(x_{k_i})| \leq |N_{G^k}(x_{k_i})|$. Since the vertex sets of all $k$-th powers of the graph are identical, for any stable set $A_{k-1}$ of $G^{k-1}$, it must hold that $|A_{k-1}|\geq |A_k|$. Thus, $\mathrm{v}(I(G^{k-1}))\leq \mathrm{v}(I(G^k))$ for $2\leq k\leq d-1$. 
\end{proof}

In particular, we calculate the $\mathrm{v}$-$\mathrm{number}$ of the edge ideals of the squares of path graphs and cycle graphs. 

\begin{Theorem}
\label{linetwo}
Let $I(L_n^2)$ be the edge ideal of the square of the path graph $L_n$. Then,
    \[
    \mathrm{v}(I(L_n^2)) =
    \begin{cases}
        [\frac{n}{6}], & \text{if}\, n \equiv 0,1 (\mathrm{mod} \, 6) \\
        [\frac{n}{6}] + 1. & \text{otherwise} 
    \end{cases}
    \]
\end{Theorem}
\begin{proof}
    We prove the theorem by considering the following cases.

    \textbf{Case 1:} Let $n = 6t(t\in \mathbb{N})$. We proceed by induction on $t$. For $t=1$, we have $\mathrm{v}(I(L_6^2)) = 1$. $\mathrm{v}(I(L_6^2))$. The stable set corresponding to $\mathrm{v}(I(L_6^2))$ is $A_6 = \{ x_4 \}$, and $N_{L_6^2}(A_6) = \{ x_2, x_3, x_5, x_6 \}$. 

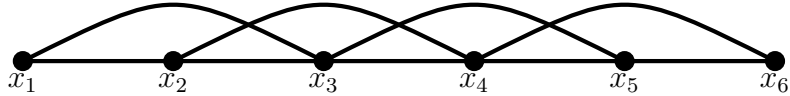
    \begin{center}
    \begin{tikzpicture}

        % 绘制六个点
        \filldraw[black, ultra thick] (0,0) circle (3pt) node[below] {$x_1$};
        \filldraw[black, ultra thick] (2,0) circle (3pt) node[below] {$x_2$};
        \filldraw[black, ultra thick] (4,0) circle (3pt) node[below] {$x_3$};
        \filldraw[black, ultra thick] (6,0) circle (3pt) node[below] {$x_4$};
        \filldraw[black, ultra thick] (8,0) circle (3pt) node[below] {$x_5$};
        \filldraw[black, ultra thick] (10,0) circle (3pt) node[below] {$x_6$};
        
        %使用直线连接点
        \draw[ultra thick] (0,0) -- (10,0);

        % 使用弧线连接点
        \draw[ultra thick] (0,0) .. controls (2,1) and (2,1) .. (4,0);  % x1 to x3
        \draw[ultra thick] (2,0) .. controls (4,1) and (4,1) .. (6,0);  % x2 to x4
        \draw[ultra thick] (4,0) .. controls (6,1) and (6,1) .. (8,0);  % x3 to x5
        \draw[ultra thick] (6,0) .. controls (8,1) and (8,1) .. (10,0); % x4 to x6

    \end{tikzpicture}
    \end{center}
    Assume the statement holds for $n=6t$. We now prove the statement for $n=6(t+1)$. By the induction hypothesis, $\mathrm{v}(I(L_{6t}^2)) = t$ corresponds to the stable set $A_{6t} = \{ x_4, x_{10}, \dots, x_{6t-2} \}$. 
    
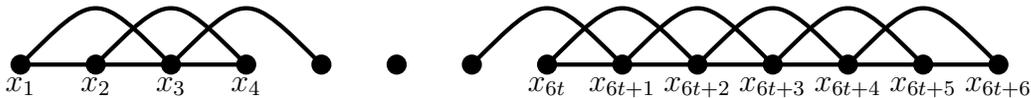
    \begin{center}
        \begin{tikzpicture}
    
            \filldraw[black, ultra thick] (0,0) circle (3pt) node[below] {$x_1$};
            \filldraw[black, ultra thick] (1,0) circle (3pt) node[below] {$x_2$};
            \filldraw[black, ultra thick] (2,0) circle (3pt) node[below] {$x_3$};
            \filldraw[black, ultra thick] (3,0) circle (3pt) node[below] {$x_4$};
            \filldraw[black, ultra thick] (4,0) circle (3pt) node[below] {};
            \filldraw[black, ultra thick] (5,0) circle (3pt) node[below] {};
            \filldraw[black, ultra thick] (6,0) circle (3pt) node[below] {};
            \filldraw[black, ultra thick] (7,0) circle (3pt) node[below] {$x_{6t}$};
            \filldraw[black, ultra thick] (8,0) circle (3pt) node[below] {$x_{6t+1}$};
            \filldraw[black, ultra thick] (9,0) circle (3pt) node[below] {$x_{6t+2}$};
            \filldraw[black, ultra thick] (10,0) circle (3pt) node[below] {$x_{6t+3}$};
            \filldraw[black, ultra thick] (11,0) circle (3pt) node[below] {$x_{6t+4}$};
            \filldraw[black, ultra thick] (12,0) circle (3pt) node[below] {$x_{6t+5}$};
            \filldraw[black, ultra thick] (13,0) circle (3pt) node[below] {$x_{6t+6}$};

            %使用直线连接点
            \draw[ultra thick] (0,0) -- (3,0);
            \draw[ultra thick] (7,0) -- (13,0);
    
            % 使用弧线连接点
            \draw[ultra thick] (0,0) .. controls (1,1) and (1,1) .. (2,0);  % x1 to x3
            \draw[ultra thick] (1,0) .. controls (2,1) and (2,1) .. (3,0);  % x2 to x4
            \draw[ultra thick] (2,0) .. controls (3,1) and (3,1) .. (4,0);  % x3 to x5
            \draw[ultra thick] (6,0) .. controls (7,1) and (7,1) .. (8,0);
            \draw[ultra thick] (7,0) .. controls (8,1) and (8,1) .. (9,0);
            \draw[ultra thick] (8,0) .. controls (9,1) and (9,1) .. (10,0);
            \draw[ultra thick] (9,0) .. controls (10,1) and (10,1) .. (11,0);
            \draw[ultra thick] (10,0) .. controls (11,1) and (11,1) .. (12,0);
            \draw[ultra thick] (11,0) .. controls (12,1) and (12,1) .. (13,0);

        \end{tikzpicture}
    \end{center}

    Now assume $n=6(t+1)$. Let $B$ be a stable set of $L_{6t+6}^2$ such that $N_{L_{6t+6}^2}(B)$ is a vertex cover of $L_{6t+6}^2$, and the number of vertices in $B$ is $|B| = s$. From the figure below, it can be observed that for any $1\leq i\leq 6t+6$, $N_{L_{6t+6}^2}(x_i)$ covers at most 13 edges, where $N_{L_{6t+6}^2}(x_i) = \{ x_{i-2}, x_{i-1}, x_{i+1}, x_{i+2} \}$. 

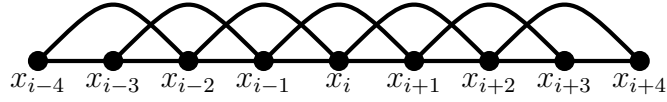
    \begin{center}
        \begin{tikzpicture}
    
            \filldraw[black, ultra thick] (0,0) circle (3pt) node[below] {$x_{i-4}$};
            \filldraw[black, ultra thick] (1,0) circle (3pt) node[below] {$x_{i-3}$};
            \filldraw[black, ultra thick] (2,0) circle (3pt) node[below] {$x_{i-2}$};
            \filldraw[black, ultra thick] (3,0) circle (3pt) node[below] {$x_{i-1}$};
            \filldraw[black, ultra thick] (4,0) circle (3pt) node[below] {$x_{i}$};
            \filldraw[black, ultra thick] (5,0) circle (3pt) node[below] {$x_{i+1}$};
            \filldraw[black, ultra thick] (6,0) circle (3pt) node[below] {$x_{i+2}$};
            \filldraw[black, ultra thick] (7,0) circle (3pt) node[below] {$x_{i+3}$};
            \filldraw[black, ultra thick] (8,0) circle (3pt) node[below] {$x_{i+4}$};
            
            %使用直线连接点
            \draw[ultra thick] (0,0) -- (8,0);
    
            % 使用弧线连接点
            \draw[ultra thick] (0,0) .. controls (1,1) and (1,1) .. (2,0);  % x1 to x3
            \draw[ultra thick] (1,0) .. controls (2,1) and (2,1) .. (3,0);  % x2 to x4
            \draw[ultra thick] (2,0) .. controls (3,1) and (3,1) .. (4,0);  % x3 to x5
            \draw[ultra thick] (3,0) .. controls (4,1) and (4,1) .. (5,0);
            \draw[ultra thick] (4,0) .. controls (5,1) and (5,1) .. (6,0);
            \draw[ultra thick] (5,0) .. controls (6,1) and (6,1) .. (7,0);
            \draw[ultra thick] (6,0) .. controls (7,1) and (7,1) .. (8,0);

        \end{tikzpicture}
    \end{center}

    Since the edge $x_{i+3}x_{i+4}$ cannot be covered by the aforementioned neighborhood set, at least one of the two vertices of this edge must belong to the neighborhood set of another point in the stable set $B$ . This implies that the 13 edges covered by $N_{L_{6t+6}^2}(x_i)$ will overlap, and at least one edge will be considered repeatedly. Therefore, we consider that a neighborhood set covers at most 12 edges, leading to $12s\geq 12t+9$, which implies $s\geq t+1 > t$. Thus, $\mathrm{v}(I(L_{6t+6}^2)) > t$. By the induction hypothesis, $A_{6t} \cup \{ x_{6t+4} \}$ is a stable set, and $N_{L_{6t+6}^2}(A_{6t} \cup \{ x_{6t+4} \})$ is a vertex cover of $L_{6t+6}^2$. Therefore, $\mathrm{v}(I(L_{6t+6}^2)) \leq t+1$, and hence $\mathrm{v}(I(L_{6t+6}^2)) = t+1$. The corresponding stable set is $A_{6t+6} = \{ x_4, x_{10}, \dots, x_{6t-2} ,x_{6t+4} \}$. 

    Using a similar proof, we have: \\
\textbf{Case 2:} Let $n=6t+1$. Then, $\mathrm{v}(I(L_n^2)) = [\frac{n}{6}] $, and the corresponding stable set is $A_n = \{ x_4 ,x_{10} ,\dots ,x_{6t-2}  \}$;\\
\textbf{Case 3:} Let $n=6t+2$. Then, $\mathrm{v}(I(L_n^2)) = [\frac{n}{6}] + 1 $, and the corresponding stable set is $A_n = \{ x_4 ,x_{10} ,\dots ,x_{6t+2}  \}$;\\
\textbf{Case 4:} Let $n=6t+3$. Then, $\mathrm{v}(I(L_n^2)) = [\frac{n}{6}] + 1 $, and the corresponding stable set is $A_n = \{ x_4 ,x_{10} ,\dots ,x_{6t+3}  \}$;\\
\textbf{Case 5:} Let $n=6t+4$. Then, $\mathrm{v}(I(L_n^2)) = [\frac{n}{6}] + 1 $, and the corresponding stable set is $A_n = \{ x_4 ,x_{10} ,\dots ,x_{6t+4}  \}$;\\
\textbf{Case 6:} Let $n=6t+5$. Then, $\mathrm{v}(I(L_n^2)) = [\frac{n}{6}] + 1 $, and the corresponding stable set is $A_n = \{ x_4 ,x_{10} ,\dots ,x_{6t+4}  \}$. 

\end{proof}

\begin{Lemma}
    Let $L_n^2$ be the square of a path graph with $n$ vertices $(x_1,\dots ,x_n)$. Then, any stable set corresponding to $\mathrm{v}(I(L_n^2))$ does not simultaneously contain both endpoints of the path graph. 
\end{Lemma}
\begin{proof}
    Assume there exists a stable set $A$ corresponding to $\mathrm{v}(I(L_n^2))$ that contains both endpoints $x_1$ and $x_n$. It is clear that $N_{L_n^2}(\{ x_1 \} ) = \{ x_2, x_3 \}$ and $N_{L_n^2}(\{ x_n \} ) = \{ x_{n-1}, x_{n-2}\}$, and each neighborhood set covers three edges. In this case, we obtain $L_{n-6}$, whose vertex set is $V(L_{n-4}^2) = \{ x_4, x_5, \dots, x_{n-3} \}$. Therefore, $|A|\geq \mathrm{v}(L_{n-6}^2)$. Let $\mathrm{v}(L_{n-6}^2) = |A'|$, where $A'\subset \{ x_4, x_5, \dots, x_{n-3}\}$ is the stable set corresponding to $\mathrm{v}(L_{n-6}^2)$. Define $A''=\{ x_{i+3}\in V(L_n^2) | x_i\in A' \}$, then $|A''| = |A'| = \mathrm{v}(L_{n-6}^2)$. It can be observed that $\{ x_4 \} \cup A''$ is a stable set of $L_n^2$, and $N_{L_n^2}(A''\cup \{ x_3 \})$ is a vertex cover of $L_n^2$. Thus, $\mathrm{v}(L_n) \leq 1+\mathrm{v}(L_{n-6}) < |A|$, which contradicts the assumption. 
    
    In summary, the stable set $A$ does not simultaneously contain both endpoints of the path graph.

\end{proof}

\begin{Proposition}
    Let $L_n$ be a path graph and $C_n$ be a cycle graph. Then, for $n\geq 7$, we have 
    $$\mathrm{v}(I(C_n^2)) = \mathrm{v}(I(L_{n-5}^2)) + 1. $$
\end{Proposition}
\begin{proof}
    For $n\geq 5$, let $V(C_n^2) = \{ x_1, x_2, \dots, x_n \}$ and $V(L_{n-5}^2) = \{ x_1, x_2, \dots, x_{n-5} \}$. 

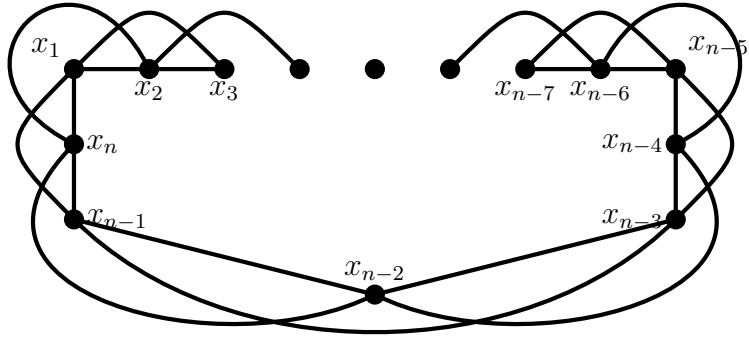
    \begin{center}
        \begin{tikzpicture}
    
            \filldraw[black, ultra thick] (0,0) circle (3pt) node[above left] {$x_1$};
            \filldraw[black, ultra thick] (1,0) circle (3pt) node[below] {$x_2$};
            \filldraw[black, ultra thick] (2,0) circle (3pt) node[below] {$x_3$};
            \filldraw[black, ultra thick] (3,0) circle (3pt) node[below] {};
            \filldraw[black, ultra thick] (4,0) circle (3pt) node[below] {};
            \filldraw[black, ultra thick] (5,0) circle (3pt) node[below] {};
            \filldraw[black, ultra thick] (6,0) circle (3pt) node[below] {$x_{n-7}$};
            \filldraw[black, ultra thick] (7,0) circle (3pt) node[below] {$x_{n-6}$};
            \filldraw[black, ultra thick] (8,0) circle (3pt) node[above right] {$x_{n-5}$};
            \filldraw[black, ultra thick] (8,-1) circle (3pt) node[left] {$x_{n-4}$};
            \filldraw[black, ultra thick] (8,-2) circle (3pt) node[left] {$x_{n-3}$};
            \filldraw[black, ultra thick] (4,-3) circle (3pt) node[above] {$x_{n-2}$};
            \filldraw[black, ultra thick] (0,-2) circle (3pt) node[right] {$x_{n-1}$};
            \filldraw[black, ultra thick] (0,-1) circle (3pt) node[right] {$x_{n}$};
            
            %使用直线连接点
            \draw[ultra thick] (0,0) -- (2,0);
            \draw[ultra thick] (6,0) -- (8,0);
            \draw[ultra thick] (8,0) -- (8,-2);
            \draw[ultra thick] (0,0) -- (0,-2);
            \draw[ultra thick] (0,-2) -- (4,-3);
            \draw[ultra thick] (8,-2) -- (4,-3);
    
            % 使用弧线连接点
            \draw[ultra thick] (0,0) .. controls (1,1) and (1,1) .. (2,0);
            \draw[ultra thick] (1,0) .. controls (2,1) and (2,1) .. (3,0);
            \draw[ultra thick] (5,0) .. controls (6,1) and (6,1) .. (7,0);
            \draw[ultra thick] (6,0) .. controls (7,1) and (7,1) .. (8,0);
            \draw[ultra thick] (1,0) .. controls (0,2) and (-2,0) .. (0,-1);
            \draw[ultra thick] (7,0) .. controls (8,2) and (10,0) .. (8,-1);
            \draw[ultra thick] (0,0) .. controls (-1,-1) and (-1,-1) .. (0,-2);
            \draw[ultra thick] (8,0) .. controls (9,-1) and (9,-1) .. (8,-2);
            \draw[ultra thick] (0,-1) .. controls (-2,-3) and (2,-4) .. (4,-3);
            \draw[ultra thick] (8,-1) .. controls (10,-3) and (6,-4) .. (4,-3);
            \draw[ultra thick] (0,-2) .. controls (2,-4) and (6,-4) .. (8,-2);

        \end{tikzpicture}
    \end{center}
    Assume $\mathrm{v}(I(L_{n-5}^2)) = |A|$, where $A$ is the stable set corresponding to the $\mathrm{v}$-$\mathrm{number}$. Since no stable set can simultaneously include both endpoints, we have $\mathrm{v}(I(L_{n-5}^2)) < \mathrm{v}(I(C_n^2))$. Let $A\cup \{ x_{n-2} \}$ be a stable set of $C_n^2$, whose neighborhood set is a vertex cover of $C_n^2$. Therefore, for any $n\geq 5 $, we have $\mathrm{v}(I(C_n^2))\leq \mathrm{v}(I(L_{n-5}^2)) + 1$. In summary, 
    $$\mathrm{v}(I(C_n^2)) = \mathrm{v}(I(L_{n-5}^2)) + 1,(n\geq 5), $$
    and the corresponding stable set is $A\cup \{ x_{n-2} \}$. 
\end{proof}

\begin{Corollary}
    Let $I(C_n^2)$ be the edge ideal of the square of the cycle graph $C_n$. Then,
    \[
    \mathrm{v}(I(C_n^2)) =
    \begin{cases}
        [\frac{n-5}{6}] + 1, & \text{if}\,  n \equiv 0,5 (\mathrm{mod} \, 6) \\
        [\frac{n-5}{6}] + 2. & \text{otherwise} 
    \end{cases}
    \]
\end{Corollary}

\section{The $\mathrm{v}$-number of various powers of an ideal}

\subsection{The v-number of monomial ideals with linear powers}
Antonino Ficarra proposed an open conjecture regarding the $\mathrm{v}$-$\mathrm{number}$ of the $k$-th power of monomial ideals \cite{ficarra2023simonconjecturetextvnumbermonomial}.

\noindent
\textbf{Conjecture}:\, Let $I\subseteq S$ be a monomial ideal with linear powers. Then, for any $k\geq 1$, 
$$\mathrm{v}(I^k) = \alpha(I)k - 1. $$
We prove that when $I$ is a homogeneous monomial ideal with linear powers and $I^k$ has no embedded associated primes for any $k\geq 1$, the equality $\mathrm{v}(I^k) = \alpha(I)k - 1$ holds. First, we prove the following proposition. 

\begin{Proposition}
\label{d-1}
    Let $I$ be an ideal generated by homogeneous polynomials of degree $d$ in the polynomial ring $S = K[x_1, x_2, \dots, x_n]$ with a linear resolution. Then, $(I : \mathfrak{m})/I$ has generators of degree $d - 1$. 
\end{Proposition}
\begin{proof}
    First, consider the linear resolution of $S/I$: 
$$0 \rightarrow F_n \rightarrow F_{n-1} \rightarrow \dots \rightarrow F_2 \rightarrow F_1 \rightarrow F_0 \rightarrow S/I \rightarrow 0, $$
where, $F_i = S(-d-(i-1))^{\beta_{n,d+(i-1)}}$ for $1\leq i\leq n$, and $F_0 = S$. Tensoring this long exact sequence with $S/\mathfrak{m}$ , we obtain: 
$$0 \rightarrow K(-d-(n-1))^{\beta_{n, d+n-1}} \rightarrow K(-d-(n-2))^{\beta_{n, d+n-2}} \rightarrow \dots$$ 
$$\rightarrow K(-d)^{\beta_{1,d}} \rightarrow K \rightarrow S/I\otimes S/\mathfrak{m} \rightarrow 0, $$
where $K = S/\mathfrak{m}$, and $K(-d-(n-1))^{\beta_{n, d+n-1}} \cong Tor_n(S/\mathfrak{m} , S/I)$. 
    
Similarly, the linear resolution of $S/\mathfrak{m}$ is: 
$$0 \rightarrow S_{x_1x_2\dots x_n}(-n) \rightarrow \overset{n}{\underset{i = 1}{\oplus}} S_{x_1\dots \hat{x_i}\dots x_n}(-n+1) \rightarrow \dots $$
$$\rightarrow \overset{n}{\underset{i = 1}{\oplus}} S_{x_i}(-1) \rightarrow S \rightarrow S/\mathfrak{m} \rightarrow 0,$$
Tensoring this long exact sequence with $S/I$, we obtain: 
$$0 \rightarrow S_{x_1x_2\dots x_n}(-n) \otimes S/I \rightarrow \overset{n}{\underset{i = 1}{\oplus}} S_{x_1\dots \hat{x_i}\dots x_n}(-n+1)\otimes S/I \rightarrow \dots $$
$$\rightarrow \overset{n}{\underset{i = 1}{\oplus}} S_{x_i}(-1) \otimes S/I \rightarrow S\otimes S/I \rightarrow S/\mathfrak{m}\otimes S/I \rightarrow 0,$$
where $S_{x_1x_2\dots x_n}(-n) \otimes S/I \cong (I:\mathfrak{m})/I(-n) \cong Tor_n(S/\mathfrak{m} , S/I)$. Therefore, 
$$(I:\mathfrak{m})/I(-n) \cong K(-d-(n-1))^{\beta_{n, d+n-1}}. $$

In summary, $(I : \mathfrak{m})/I$ has generators of degree $d - 1$. 
\end{proof}

\begin{Corollary}
\label{dk-1}
    Let $I$ be an ideal generated by homogeneous polynomials of degree $d$ in the polynomial ring $S = K[x_1, x_2, \dots, x_n]$ with linear powers. Then, for any $k\geq 1$, $(I^k : \mathfrak{m})/I^k$ has generators of degree $kd-1$. 
\end{Corollary}

\begin{Theorem}
\label{1-d}
    Let $I$ be an ideal generated by homogeneous monomials of degree $d$ in the polynomial ring $S = K[x_1, x_2, \dots, x_n]$ with a linear resolution, and suppose $I$ has no embedded associated primes. Then, 
    $$\mathrm{v}(I) = d - 1. $$
\end{Theorem}
\begin{proof}
    By Lemma \ref{alpha}, since the monomial ideal $I$ has no embedded primes, we have 
    $$\mathrm{v}(I) = \mathrm{min}\{ \alpha((I:\mathfrak{p})/I) | \mathfrak{p}\in \mathrm{Ass}(I) \}. $$
    By Proposition \ref{d-1}, and noting that $(I:\mathfrak{m}) \subseteq (I:\mathfrak{p})$, it follows that $\mathrm{v}(I)\leq d-1$. By the definition of the $\mathrm{v}$-$\mathrm{number}$, we have $\mathrm{v}(I)\geq \alpha(I)-1 = d-1$. Therefore, $\mathrm{v}(I) = d-1$. 
\end{proof}

Further considering the monomial ideal $I^k(k\geq 1)$ in the conjecture, we have the following result. 

\begin{Theorem}
\label{kd-1}
Let $I$ be a homogeneous monomial ideal of degree $d$ in the polynomial ring $S = K[x_1, x_2, \dots, x_n]$ with linear powers, and suppose that for any $k \geq 1$, the ideal $I^k$ has no embedded associated primes. Then, 
    $$\mathrm{v}(I^k) = \alpha(I)k-1. $$
\end{Theorem}

The reference \cite{simis1994ideal} proves that for any power $I^k$($k\geq 1$) of the edge ideal of a bipartite graph, there are no embedded associated primes. Therefore, combining with Theorem \ref{kd-1}, if the edge ideal $I$ of a bipartite graph has linear powers, we have the following result. 

\begin{Corollary}
    Let $I$ be the edge ideal of a bipartite graph with linear powers. Then, for any $k \geq 1$, $\mathrm{v}(I^k) = 2k-1$. 
\end{Corollary}

\subsection{The v-number of the k-th power of an ideal}
The $\mathrm{v}$-$\mathrm{number}$ of powers of the edge ideal of a path graph is known to exhibit the following two cases. 

\begin{Example}
    Let the monomial ideal $I(L_n)$ be the edge ideal of the path graph $L_n$, where $n \equiv 2,3 (\mathrm{mod} 4)$. Without loss of generality, let $n=4k+2$ or $n=4k+3 (k\geq 1)$. Take the monomial corresponding to $\mathrm{v}(I(L_n))$ as $f = x_3x_7\dots x_{4k-1}x_{4k+1}$. Then, the monomial corresponding to $\mathrm{v}(I(L_n)^2)$ can be chosen as 
    $$g = x_3x_7\dots x_{4k-1}x_{4k}x_{4k+1},$$
    and thus $\mathrm{v}(I(L_n)^2) = \mathrm{v}(I(L_n)) + 1$. 
\end{Example}
\begin{Example}
    Let $I$ be the edge ideal of the path graph $L_8$. Then, $\mathrm{v}(I(L_8))=2$ and $\mathrm{v}(I(L_8)^2) = 4 = \mathrm{v}(I(L_8)) + 2$. 
\end{Example}

It can be observed that the difference in the $\mathrm{v}$-$\mathrm{number}$ between consecutive powers of the edge ideal of a path graph exhibits two cases. The following theorem provides a detailed explanation of the relationship between the $\mathrm{v}$-$\mathrm{number}$ of powers of the edge ideal $I(G)$ of an arbitrary graph. 

\begin{Theorem}
\label{IGk}
Let the monomial ideal $I = I(G)$ be the edge ideal of a graph $G$. Then, for any $k \geq 1$, we have 
    $$\mathrm{v}(I^k) + 2 \geq \mathrm{v}(I^{k+1}) \geq \mathrm{v}(I^k) + 1. $$
\end{Theorem}
\begin{proof}
    First, we prove $\mathrm{v}(I^k) + 2 \geq \mathrm{v}(I^{k+1})$. Let $f$ be the monomial corresponding to the $\mathrm{v}$-$\mathrm{number}$ of the ideal $I^k$. Then, there exists an associated prime ideal $\mathfrak{p}$ of $I^k$ such that $(I^k : f) = \mathfrak{p}$. Without loss of generality, let the generators of the edge ideal $I$ be $\mu_1, \mu_2,\dots, \mu_m$. Then, 
    $$\mathfrak{p} = (I^k : f) = ((I^{k+1} : I) : f) = (I^{k+1} : If) = \mathrel{\mathop{\cap}\limits_{\substack{i=1}{}}^{\substack{m}{}}} (I^{k+1} : \mu_i f), $$
    where there exists some $i (1\leq i\leq m)$ such that $\mathfrak{p} = (I^{k+1} : \mu_i f)$. Since $\mathfrak{p} \in \mathrm{Ass}(I^k) \subset \mathrm{Ass}(I^{k+1})$, we have 
    $$\mathrm{v}(I^k) + 2 \geq \mathrm{v}(I^{k+1}).$$

    Next, we prove \( \mathrm{v}(I^{k+1}) \geq \mathrm{v}(I^k) + 1 \). Let \( g \) be the monomial corresponding to the \(\mathrm{v}\)-\(\mathrm{number}\) of the ideal \( I^{k+1} \). Then, there exists an associated prime ideal \( \mathfrak{q} \) of \( I^{k+1} \) such that \( (I^{k+1} : g) = \mathfrak{q} \). Therefore,
    \[
    g \in (I^{k+1} : \mathfrak{q}) \setminus I^{k+1}, \quad g \in I^k,
    \]
    and for any \( x \in \mathfrak{q} \), we have \( gx \in I^{k+1} \). Let \( g = \mu_1 \mu_2 \dots \mu_i \dots \mu_r v \), where \( \mu_i \) (\( 1 \leq i \leq r \leq k \)) are generators of the edge ideal \( I \). Consider the following two cases: 

    \textbf{Case 1:} If for any \( x \in \mathfrak{q} \), \( x \) is directly connected to some vertex in \( v \), then take \( g' = \frac{g}{\mu_1} \); 

    \textbf{Case 2:} If there exists some \( x_i \in \mathfrak{q} \) that is connected to an element \( v_i \) in \( v \) through \( \mu_{i1}, \dots, \mu_{is} \) (\( s \geq 1 \)), and let \( y \) be the vertex in \( \mu_{i1}, \dots, \mu_{is} \) connected to \( v \), then take \( g' = \frac{g}{y} \).

    For \( g' \) in both cases, we have: for any \( x \in \mathfrak{q} \), \( xg' \in I^k \). Therefore,
    \[
    g' \in (I^k : \mathfrak{q}) \setminus I^k, \quad g' \in I^{k-1}.
    \]
    Additionally, for any \( x \notin \mathfrak{q} \), if \( x \in (I^k : g') \), then \( x \) is connected to the vertices in \( g' \) through an even path, and thus \( x \) is connected to the vertices in \( g \) through an even path, i.e., \( xg \in I^{k+1} \). Therefore, \( x \in \mathfrak{q} \). In conclusion, \( \mathfrak{q} = (I^k : g') \), and thus
    \[
    \mathrm{v}(I^{k+1}) \geq \mathrm{deg}(g') + 1 \geq \mathrm{v}_{\mathfrak{q}}(I^k) + 1 \geq \mathrm{v}(I^k) + 1.
    \]

    In summary, \( \mathrm{v}(I^k) + 2 \geq \mathrm{v}(I^{k+1}) \geq \mathrm{v}(I^k) + 1 \).
\end{proof}

\begin{Theorem}
\label{I^k+2}
Let $I$ be a homogeneous monomial ideal of degree $d$ in a Noetherian integral domain $S$. Then, for $k\gg 0$, we have 
$$\mathrm{v}(I^{k+1}) = \mathrm{v}(I^k) + d.$$
\end{Theorem}
\begin{proof}
Let $\mathfrak{p}$ be an associated prime ideal of $I^{k+1}$ such that $(I^{k+1}:f) = \mathfrak{p}$, and $\mathrm{deg}(f) = \mathrm{v}(I^{k+1})$. Since for a graded ideal $I$, when $k\gg 0$, we have $\mathrm{Ass}(I) = \mathrm{Ass}(I^k)$, it follows that $\mathrm{Ass}^{\infty}(I) = \mathrm{Ass}(I)$. By Lemma \ref{inftyprime}, for $k\gg 0$ and any $ \mathfrak{p} \in \mathrm{Ass}(I)$, we have
$$(I^{k+1}:\mathfrak{p}) = I(I^k:\mathfrak{p}). $$

Let \( f = \mu g \), where \( \mu \in I \), and thus \( g \in (I^k : \mathfrak{p}) \setminus I^k \). Therefore,
\[
\mathfrak{p} \subseteq (I^k : g) \subseteq (\mu I^k : \mu g) = (\mu I^k : f) \subseteq (I^{k+1} : f) = \mathfrak{p}.
\]
Hence, \( (I^k : g) = \mathfrak{p} \). Furthermore,
\[
\mathrm{v}(I^{k+1}) = \mathrm{deg}(f) = \mathrm{deg}(g) + \mathrm{deg}(\mu) \geq \mathrm{v}_{\mathfrak{p}}(I^k) + d \geq \mathrm{v}(I^k) + d.
\]

Let \( \mathfrak{q} \) be an associated prime ideal of \( I^k \) such that \( (I^k : g) = \mathfrak{q} \), and \( \mathrm{deg}(g) = \mathrm{v}(I^k) \). Since \( S \) is a Noetherian integral domain, let the generators of \( I \) be \( \mu_1, \mu_2, \dots,\) \( \mu_r \). Then,
\[
\mathfrak{q} = ((I^{k+1} : I) : g) = (I^{k+1} : Ig) = \bigcap_{i=1}^r (I^{k+1} : g \mu_i),
\]
and thus there exists some \( j \) (\( 1 \leq j \leq r \)) such that \( \mathfrak{q} = (I^{k+1} : g \mu_j) \). Since for \( k \gg 0 \), we have \( \mathrm{Ass}(I^k) = \mathrm{Ass}(I^{k+1}) = \mathrm{Ass}^{\infty}(I) \), it follows that
\[
\mathrm{v}(I^{k+1}) \leq \mathrm{v}_{\mathfrak{q}}(I^{k+1}) \leq \mathrm{v}(I^k) + d.
\]

In summary, the result holds. 
\end{proof}

Let \( I(G) \subseteq S = K[x_1, \dots, x_n] \) be the edge ideal of a graph \( G \). According to the new graph \( H \) defined in Lemma \ref{H}: \( V(H) = V(G) \setminus \mathrm{supp}(\mu) \), and \( I(H) = (I(G)^{[k+1]} : \mu) \) is the edge ideal of the graph \( H \), where \( \mu = e_1 \dots e_k \in I(G)^{[k]} \). Here, \( \mathrm{supp}(\mu) \) denotes the support of the monomial \( \mu \), i.e., the set of all variables dividing \( \mu \), and \( I^{[k]} \) denotes the \( k \)-th square-free power of the ideal \( I \), which is generated by all square-free monomials in \( I^k \) (see reference \cite{gu2017regularity}). Then, for any \( 1 \leq k \leq \mathrm{match}(G) \), we have
\[
\mathrm{v}(I(G)^{[k+1]}) \leq \mathrm{v}(I(H)) + 2k,
\]
where \( \mathrm{match}(G) \) denotes the matching number of the graph \( G \), i.e., the size of the maximum matching in \( G \).

\begin{Lemma}[\cite{jaramillo2021v}]
\label{ta}
Let \( I = I(\mathcal{C}) \) be the edge ideal of a Clutter graph \( \mathcal{C} \). If for a monomial \( f \in S_d \) and a prime ideal \( \mathfrak{p} \in \mathrm{Ass}(I) \), we have \( (I : f) = \mathfrak{p} \), then there exists a stable set \( A \in \mathcal{A}_{\mathcal{C}} \) with \( |A| \leq d \) such that
\[
(I : t_A) = (N_{\mathcal{C}}(A)) = \mathfrak{p},
\]
where \( \mathcal{A}_{\mathcal{C}} \) is the collection of all stable sets in the Clutter graph \( \mathcal{C} \) whose neighborhood sets are minimal vertex covers, \( N_{\mathcal{C}}(A) \) is the neighborhood set of the stable set \( A \), and \( t_A \) is the monomial formed by the product of all variables in \( A \).
\end{Lemma}

\begin{Theorem}
\label{I^[k]}
Let \( I = I(L_n) \) be the edge ideal of the path graph \( L_n \). For any \( 1 \leq k \leq \mathrm{match}(L_n) \), we have:
\begin{itemize}
    \item[(1)] When \( k \) is odd,
    \[
    \mathrm{v}(I^{[k]}) = 
    \begin{cases}
        [\frac{n}{4}] + \frac{3(k-1)}{2}, & \text{if } n \equiv 0,1 \ (\mathrm{mod} \ 4) \\
        [\frac{n}{4}] + \frac{3(k-1)}{2} + 1, & \text{if } n \equiv 2,3 \ (\mathrm{mod} \ 4).
    \end{cases}
    \]
    \item[(2)] When \( k \) is even,
    \[
    \mathrm{v}(I^{[k]}) = [\frac{n}{4}] + \frac{3k}{2} - 1.
    \]
\end{itemize}
\end{Theorem}
\begin{proof}
Let \( f \) be the monomial corresponding to the \(\mathrm{v}\)-\(\mathrm{number}\) of \( I^{[k]} \), i.e., \( \mathrm{v}(I^{[k]}) = \mathrm{deg}(f) \) and \( (I^{[k]} : f) = \mathfrak{p} = \langle x_{i_1}, \dots, x_{i_r} \rangle \), where \( \mathfrak{p} \in \mathrm{Ass}(I^{[k]}) \). The ideal \( I^{[k]} \) can be viewed as the edge ideal of a Clutter graph \( \mathcal{C} \). By Lemma \ref{ta}, there exists a stable set \( A \) of \( \mathcal{C} \) such that \( t_A = f \) and \( N_{\mathcal{C}}(A) = \mathfrak{p} \). Combining the definition of the quotient ideal, \( f \) must contain at least \( k-1 \) distinct generators of \( I \); otherwise, it would not yield a single element in the associated prime ideal \( \mathfrak{p} \). Without loss of generality, let the generators of the edge ideal of \( L_n \) be \( \mu_1, \dots, \mu_{n-1} \), where \( \mu_i = x_i x_{i+1} \) (\( 1 \leq i \leq n-1 \)). Let \( f = f' u \), where \( u \) is the product of any \( k-1 \) distinct generators of \( L_n \).

By Lemma \ref{H}, we have \( (I^{[k]} : u) = I(H) \), and the graph \( H \) is isomorphic to \( L_{n-2(k-1)} \). Then,
    \[
    (I^{[k]} : f) = (I^{[k]} : f' u) = (I^{[k]} : u) : f' = (I(H) : f') = \mathfrak{p},
    \]
where the generating set of \( \mathfrak{p} \) is a minimal vertex cover of the Clutter graph \( \mathcal{C} \) and thus also a minimal vertex cover of \( H \). Therefore, \( \mathfrak{p} \) is an associated prime ideal of \( I(H) \). Hence,
    \[
    \mathrm{deg}(f') \geq \mathrm{v}(I(H)) = \mathrm{v}(I(L_{n-2(k-1)})).
    \]
Combining with Lemma \ref{line}, we obtain
    \[
    \mathrm{v}(I^{[k]}) = \mathrm{deg}(f) = \mathrm{deg}(f') + 2(k-1),
    \]
    hence,
    \[
    \mathrm{v}(I^{[k]}) \geq
    \begin{cases}
        [\frac{n-2(k-1)}{4}], & \text{if } (n-2(k-1)) \equiv 0,1 \ (\mathrm{mod} \ 4) \\
        [\frac{n-2(k-1)}{4}] + 1, & \text{if } (n-2(k-1)) \equiv 2,3 \ (\mathrm{mod} \ 4).
    \end{cases}
    \]
    
Next, we simplify the above results by discussing \( k-1 \).

\textbf{Case 1.} When \( k-1 \) is even, based on \( n-2(k-1) \equiv 0,1 \ (\mathrm{mod} \ 4) \), we have
    \[
    [\frac{n-2(k-1)}{4}] = [\frac{n}{4}] - \frac{k-1}{2} + 2(k-1) = [\frac{n}{4}] + \frac{3(k-1)}{2}.
    \]
Here, \( n \equiv 0,1 \ (\mathrm{mod} \ 4) \). Similarly, when \( n-2(k-1) \equiv 2,3 \ (\mathrm{mod} \ 4) \), we have
    \[
    [\frac{n-2(k-1)}{4}]+ 1 = [\frac{n}{4}] + \frac{3(k-1)}{2} + 1.
    \]
Here, \( n \equiv 2,3 \ (\mathrm{mod} \ 4) \).
    
\textbf{Case 2.} When \( k-1 \) is odd, based on \( n-2(k-1) \equiv 0,1 \ (\mathrm{mod} \ 4) \), we have
    \[
    [\frac{n-2(k-1)}{4}] = [\frac{n}{4}] - \frac{2(k-1)-2}{4} + 2(k-1) = [\frac{n}{4}] + \frac{3k}{2} - 1.
    \]
Here, \( n \equiv 2,3 \ (\mathrm{mod} \ 4) \). Similarly, when \( n-2(k-1) \equiv 2,3 \ (\mathrm{mod} \ 4) \), we have
    \[
    [\frac{n-2(k-1)}{4}] + 1 = [\frac{n}{4}] - \frac{2(k-1)+2}{4} + 1 = [\frac{n}{4}] + \frac{3k}{2} - 1.
    \]
Here, \( n \equiv 0,1 \ (\mathrm{mod} \ 4) \).
    
From the above discussion, we obtain:
    \begin{itemize}
        \item[(1)] When \( k \) is odd,
         \[
         \mathrm{v}(I^{[k]}) \geq
        \begin{cases}
            [\frac{n}{4}] + \frac{3(k-1)}{2}, & \text{if } n \equiv 0,1 \ (\mathrm{mod} \ 4) \\
            [\frac{n}{4}] + \frac{3(k-1)}{2} + 1, & \text{if } n \equiv 2,3 \ (\mathrm{mod} \ 4).
        \end{cases}
        \]
        \item[(2)] When \( k \) is even,
        \[
        \mathrm{v}(I^{[k]}) \geq [\frac{n}{4}] + \frac{3k}{2} - 1.
        \]
    \end{itemize}
    
Next, we prove the reverse inequality. Here, we only prove the case when \( k \) is odd; the even case can be proved similarly.
    
    When \( n \equiv 0,1 \ (\mathrm{mod} \ 4) \), let \( \frac{n}{4} = m \). We can choose the monomial \( f \) corresponding to \( \mathrm{v}(I^{[k]}) \) as
    \[
    f = x_3 x_7 \dots x_{4m-1-4(\frac{k-1}{2})} x_{4m-4(\frac{k-1}{2})} x_{4m-4(\frac{k-1}{2})+1} \dots x_{4m-2} x_{4m-1},
    \]
    satisfying \( (I^{[k]} : f) = \mathfrak{p} = \langle x_2, x_4, x_6, \dots, x_{4m-2-4(\frac{k-1}{2})}, x_{4m} \rangle \in \mathrm{Ass}(I^{[k]}) \), and
    \[
    \mathrm{deg}(f) = [\frac{n}{4}] + \frac{3(k-1)}{2},
    \]
which is the desired monomial.
    
When \( n \equiv 2,3 \ (\mathrm{mod} \ 4) \), we can choose the monomial \( f \) corresponding to \( \mathrm{v}(I^{[k]}) \) as
    \[
    f = x_3 x_7 \dots x_{4m-1-4(\frac{k-2}{2})} x_{4m-4(\frac{k-2}{2})} x_{4m-4(\frac{k-2}{2})+1} \dots x_{4m} x_{4m+1},
    \]
satisfying \( (I^{[k]} : f) = \mathfrak{p}' = \langle x_2, x_4, x_6, \dots, x_{4m-2-4(\frac{k-2}{2})}, x_{4m+2} \rangle \in \mathrm{Ass}(I^{[k]}) \), and
    \[
    \mathrm{deg}(f) = [\frac{n}{4}] + \frac{3(k-1)}{2} + 1,
    \]
which is the desired monomial.
    
In summary, the result holds.
\end{proof}

\begin{Conjecture}
    Let \( I = I(L_n) \) be the edge ideal of the path graph \( L_n \). For any \( 1 \leq k \leq \mathrm{match}(L_n) \), we have:
    $$\mathrm{v}(I^k) = \mathrm{v}(I^{[k]}).$$
\end{Conjecture}

\begin{Proposition}
    Let \( G \) be a forest graph with \( n \) vertices, and let \( I(G) \) be its edge ideal. Then, for any \( 1 \leq k \leq \mathrm{match}(G) \), we have \( \mathrm{v}(I(G)^{[k]}) = 2k - 1 \).
\end{Proposition}
\begin{proof}
By Lemma \ref{forest}, \( (I(G)^{[k]} : x_n) = I(G_1)^{[k]} + x_{n-1} I(G_2)^{[k-1]} \), where
    \[
    I(G_1)^{[k]} \subseteq I(G_1)^{[k-1]} \subseteq x_{n-1} I(G_2)^{[k-1]}.
    \]
Therefore, \( (I(G)^{[k]} : x_n) = x_{n-1} I(G_2)^{[k-1]} \). Since \( \mathrm{v}(I(G_2)^{[k-1]}) \geq \alpha(I(G_2)^{[k-1]}) - 1 \), \( \mathrm{v}(x_{n-1}) = 0 \), and by Proposition \ref{IJ}, we have
\begin{align*}
    \mathrm{v}(I(G)^{[k]})
    &\leq \mathrm{v}(I(G)^{[k]}:x_n)+1\\
    &\leq \mathrm{min}\{ \alpha(x_{n-1}) + \mathrm{v}(I(G_2)^{[k-1]}), \mathrm{v}(x_{n-1}) +\alpha(I(G_2)^{[k-1]}) \} + 1\\
    &\leq \alpha(I(G_2)^{[k-1]}) + 1\\
    &= 2k-1.
\end{align*}
Additionally, since \( \mathrm{v}(I(G)^{[k]}) \geq \alpha(I(G)^{[k]}) - 1 = 2k - 1 \), the result holds.
\end{proof}
%\begin{align*}
%    \big[\frac{n-2(k-1)}{4}\big] &= [\frac{n}{4}] - \frac{k-1}{2} + 2(k-1) \\
%    &= [\frac{n}{4}] + \frac{3(k-1)}{2}.
%\end{align*}
%\begin{align*}
%    \big[\frac{n-2(k-1)}{4}\big] &= [\frac{n}{4}] - \frac{2(k-1)-2}{4} +2(k-1) \\
%    &= [\frac{n}{4}] + \frac{3k}{2} - 1.
%\end{align*}
%\begin{align*}
%    \big[\frac{n-2(k-1)}{4}\big] + 1 &= [\frac{n}{4}] - \frac{2(k-1)+2}{4} + 1 \\
%    &= [\frac{n}{4}] + \frac{3k}{2} - 1.
%\end{align*}

\subsection{The v-number of symbolic powers of an ideal}

Let \( S = K[x_1, \dots, x_n] \) be a Noetherian commutative ring over a field \( K \), and let \( I \subseteq S \) be a square-free monomial ideal. Let \( \mathrm{Min}(I) \) be the set of minimal associated prime ideals of \( I \). Then, the \( k \)-th symbolic power of \( I \) satisfies
\[
I^{(k)} = \bigcap_{P \in \mathrm{Min}(I)} P^k.
\]

In this section, we primarily use the above result to study the \(\mathrm{v}\)-\(\mathrm{number}\) of the \( k \)-th symbolic power of square-free monomial ideals.

\begin{Proposition}
\label{pp}
    Let \( I \) be a square-free monomial ideal with minimal associated prime ideals \( \mathfrak{p}_i \) (where \( 1 \leq i \leq m \)), and let \( \mathfrak{p} \) be one of them. For \( k \geq 2 \), we have
    \[
    (I^{(k)} : \mathfrak{p}) = \bigcap_{\mathfrak{p}_i \neq \mathfrak{p}} \mathfrak{p}_i^k \cap \mathfrak{p}^{k-1},
    \]
    and\( (I : \mathfrak{p}) = \bigcap_{\mathfrak{p}_i \neq \mathfrak{p}} \mathfrak{p}_i. \)
\end{Proposition}
\begin{proof}
    Since
    \[
    (I^{(k)} : \mathfrak{p}) = \left( \bigcap_{i=1}^m \mathfrak{p}_i^k : \mathfrak{p} \right) = \bigcap_{i=1}^m (\mathfrak{p}_i^k : \mathfrak{p}),
    \]
    let \( \mathfrak{p}_i^k = (x_{i1}, x_{i2}, \dots, x_{il})^k \), and let \( \mathfrak{p} = (y_1, y_2, \dots, y_t) \). Then,
    \begin{align*}
        (\mathfrak{p}_i^k :\mathfrak{p}) &= ((x_{i1}, x_{i2},\dots,x_{il} )^k: (y_{1}, y_{2},\dots,y_{t})) \\ 
             &= \bigcap_{j=1}^t ((x_{i1}, x_{i2}, \dots, x_{il})^k : y_j). 
    \end{align*}
    If \( y_j \notin \mathfrak{p}_i \) for some \( j \), then one term in the intersection is \( \mathfrak{p}_i^k \), so \( (\mathfrak{p}_i^k : \mathfrak{p}) = \mathfrak{p}_i^k \). Therefore, when \( \mathfrak{p} \neq \mathfrak{p}_i \), we have \( (\mathfrak{p}_i^k : \mathfrak{p}) = \mathfrak{p}_i^k \). Additionally, since \( (\mathfrak{p}^k : \mathfrak{p}) = \mathfrak{p}^{k-1} \) and \( (\mathfrak{p} : \mathfrak{p}) = S \), the result holds.
\end{proof}

\begin{Theorem}
\label{symbolic}
    Let \( I \subseteq S \) be a homogeneous square-free monomial ideal with no embedded associated primes. Let \( m \) be the number of minimal associated primes of \( I \). Then, for \( k \geq 1 \),
    \[
    \mathrm{v}(I^{(k)}) + 1 \leq \mathrm{v}(I^{(k+1)}) \leq \mathrm{v}(I^{(k)}) + m.
    \]
\end{Theorem}
\begin{proof}
    Let \( I = \bigcap_{i=1}^m \mathfrak{p}_i \), so \( I^{(k)} = \bigcap_{i=1}^m \mathfrak{p}_i^k \). By the definition of associated primes, \( \mathrm{Ass}(I^{(k)}) = \mathrm{Ass}(I^{(k+1)}) \). Combining with Proposition \ref{pp}, we have
    \[
    (I^{(k+1)} : \mathfrak{p}_1 \dots \mathfrak{p}_m) = (((I^{(k+1)} : \mathfrak{p}_1) : \dots) : \mathfrak{p}_m) = I^{(k)}.
    \]
    Let \( \mathfrak{p} \) be the minimal associated prime corresponding to \( \mathrm{v}(I^{(k)}) \), and let \( \mathfrak{p}' \) be the minimal associated prime corresponding to \( \mathrm{v}(I^{(k+1)}) \). Then,
    \[
    (I^{(k)} : \mathfrak{p}) \setminus I^{(k)} = \left( (I^{(k+1)} : \mathfrak{p}') : \prod_{\mathfrak{p}_i \neq \mathfrak{p}'} \mathfrak{p}_i \mathfrak{p} \right) \setminus I^{(k)}.
    \]
    Since the degree of the terms in \( \prod_{\mathfrak{p}_i \neq \mathfrak{p}'} \mathfrak{p}_i \mathfrak{p} \) is fixed at \( m \), and \( I^{(k+1)} \subseteq I^{(k)} \), by Lemma \ref{alpha},
    \[
    \mathrm{v}(I^{(k+1)}) \leq \mathrm{v}(I^{(k)}) + m.
    \]

    Next, we prove \( \mathrm{v}(I^{(k+1)}) \geq \mathrm{v}(I^{(k)}) + 1 \). Let \( \mathfrak{p}_i = (x_{i1}, \dots, x_{i t_i}) \), where \( t_i \) is the number of generators of \( \mathfrak{p}_i \). Let \( g \in (I^{(k+1)} : \mathfrak{p}_1) \setminus I^{(k+1)} \) with \( \mathrm{deg}(g) = \mathrm{v}(I^{(k+1)}) \). Then,
    \[
    g = \mathbf{x}^{\alpha_g} \in \bigcap_{i=2}^m \mathfrak{p}_i^{k+1} \cap \mathfrak{p}_1^k,
    \]
    satisfying
    \[
    \alpha_g \cdot \mathbf{1}_{\mathfrak{p}_i} \geq k + 1 \quad (2 \leq i \leq m), \quad \alpha_g \cdot \mathbf{1}_{\mathfrak{p}_1} \geq k,
    \]
    where \( \mathbf{x}^{\alpha_g} = x_1^{\alpha_1} x_2^{\alpha_2} \dots x_n^{\alpha_n} \) is a monomial, \( \alpha_i \) is the degree of \( x_i \) in \( g \), \( \alpha_g = (\alpha_1, \dots, \alpha_n) \), and \( \mathbf{1}_{\mathfrak{p}_i} \) is an \( n \)-dimensional vector with $1$ at positions corresponding to the generators of \( \mathfrak{p}_i \) and $0$ elsewhere.

    Let \( u = x_i \in \mathfrak{p}_1 \) for some \( 1 \leq i \leq t_1 \). Define \( f = \frac{g}{u} \), and let \( f = \mathbf{x}^{\alpha_f} \), satisfying
    \[
    \alpha_f \cdot \mathbf{1}_{\mathfrak{p}_i} \geq k \quad (2 \leq i \leq m), \quad \alpha_f \cdot \mathbf{1}_{\mathfrak{p}_1} \geq k - 1.
    \]
    Therefore,
    \[
    f \in \bigcap_{i=2}^m \mathfrak{p}_i^k \cap \mathfrak{p}_1^{k-1} = (I^{(k)} : \mathfrak{p}_1),
    \]
    and
    \[
    g \notin I^{(k+1)} \iff f = \frac{g}{u} \notin I^{(k)}.
    \]
    Thus, \( f \in (I^{(k)} : \mathfrak{p}) \setminus I^{(k)} \), and we obtain
    \[
    \mathrm{v}(I^{(k+1)}) = \mathrm{deg}(g) \geq \mathrm{deg}(f) + \mathrm{deg}(u) \geq \mathrm{v}(I^{(k)}) + 1.
    \]
\end{proof}

\begin{Theorem}
\label{complete}
    Let \( G \) be a complete graph with vertex set \( V(G) = \{ x_1, \dots, x_n \} \), and let \( I \) be the edge ideal of \( G \). Then, for any \( k \geq 1 \),
    \[
    \mathrm{v}(I^{(k)}) = k + [\frac{k-1}{n-1}] + 1.
    \]
\end{Theorem}
\begin{proof}
    It is known that the \(\mathrm{v}\)-\(\mathrm{number}\) of the edge ideal of a complete graph is 1, which satisfies the formula. Now, we prove the case for \( k \geq 2 \). Let
    \[
    \mathfrak{p}_i = \langle x_1, \dots, x_{i-1}, x_{i+1}, \dots, x_n \rangle \quad (2 \leq i \leq n-1),
    \]
    \( \mathfrak{p}_1 = \langle x_2, x_3, \dots, x_n \rangle \), and \( \mathfrak{p}_n = \langle x_1, x_2, \dots, x_{n-1} \rangle \). Then, \( I^{(k)} = \bigcap_{i=1}^n \mathfrak{p}_i^k \). Since \( I \) is the edge ideal of a complete graph, it has no embedded associated primes. Consider the quotient ideal
    \[
    L_k = \frac{I^{(k)} : \mathfrak{p}_1}{I^{(k)}} = \frac{\mathfrak{p}_1^{k-1} \cap \left( \bigcap_{i=2}^n \mathfrak{p}_i^k \right)}{\bigcap_{i=1}^n \mathfrak{p}_i^k}.
    \]
    Let \( \prod_{i=1}^n x_i^{\alpha_i} + I^{(k)} \) be a nonzero element in \( L_k \). Then:
    \begin{itemize}
        \item[(1)] For any \( 2 \leq i_1 < i_2 < \dots < i_{n-2} \leq n \),
        \[
        \alpha_1 + \alpha_{i_1} + \alpha_{i_2} + \dots + \alpha_{i_{n-2}} \geq k.
        \]
        \item[(2)] \( \alpha_2 + \dots + \alpha_n = k - 1 \).
    \end{itemize}
    From the inequality
    \begin{align*}
        (n-1)\sum_{i=1}^{n} \alpha_i &= \sum_{j=2}^{n}\sum_{i\neq j} \alpha_i + (\alpha_2 + \alpha_3 + \dots +\alpha_n) \\
        &\geq (n-1)k +k-1 \\
        &= nk-1, 
    \end{align*}
    we have \( \sum_{i=1}^n \alpha_i \geq \frac{nk - 1}{n-1} \). Let \( k = m(n-1) + r + 1 \), where \( 1 \leq r < n-1 \). Then,
    \[
    k + m < \frac{nk - 1}{n-1} < k + m + 1.
    \]
    When \( r = n-1 \), the right-hand side of the inequality becomes an equality. Since \( \sum_{i=1}^n \alpha_i \) is an integer, for any \( k \geq 2 \),
    \[
    \sum_{i=1}^n \alpha_i \geq k + m + 1,
    \]
    where \( m = [\frac{k-1}{n-1}] \). Therefore, for any \( s \leq k + m \), \( (L_k)_s = 0 \), and a nonzero element in \( (L_k)_{k+m+1} \) is
    \[
    x_1^2 x_2^{m+1} x_3^{m+1} \dots x_{r+1}^{m+1} x_{r+2}^m \dots x_n^m + I^{(n)}.
    \]
    
    In summary, by Lemma \ref{alpha}, \( \mathrm{v}(I^{(k)}) = k + [\frac{k-1}{n-1}] + 1 \).
\end{proof}

\textbf{Acknowledgment:} The authors are grateful to the software systems \cite{team2004cocoa} for providing us with a large number
of examples to develop ideas and test our results.

\textbf{Statement:} On behalf of all authors, the corresponding author states that there is
 no conflflict of interest.

\newpage

\bibliographystyle{plain}
\bibliography{ref}

\end{document}